\documentclass[oneside, 11pt, english]{amsart}
\usepackage[T1]{fontenc}
\usepackage[latin9]{inputenc}
\usepackage{amsthm}
\usepackage{amstext}
\usepackage{amssymb}
\usepackage{esint}
\usepackage{graphicx}
\usepackage{enumerate}

\makeatletter
\newtheorem{theorem}{Theorem}[section]
\newtheorem{lemma}[theorem]{Lemma}
\newtheorem{proposition}[theorem]{Proposition}

\numberwithin{figure}{section}
\theoremstyle{definition}

\theoremstyle{remark}
\newtheorem{remark}[theorem]{Remark}

\numberwithin{equation}{section}
\makeatother

\usepackage{babel}

	\DeclareMathOperator{\loc}{loc}

	\DeclareMathOperator{\spn}{span}
	\DeclareMathOperator{\Imag}{Im}

\begin{document}

\title[Spectral Stability]{Spectral Stability Estimates of Neumann Divergence Form Elliptic Operators}

\author{VLADIMIR GOL'DSHTEIN, VALERII PCHELINTSEV and ALEXANDER UKHLOV$^1$}

\maketitle
\noindent
{\scriptsize We study spectral stability estimates of elliptic operators in divergence form $-\textrm{div} [A(w) \nabla g(w)]$ with the Neumann boundary condition in non-Lipschitz domains $\Omega \subset \mathbb C$. The suggested method is based on connections of planar quasiconformal mappings with Sobolev spaces and its applications to the Poincar\'e inequalities.}



\section{\textbf{INTRODUCTION}}

In this paper we give applications of the quasiconformal mapping theory to the spectral stability estimates of the Neumann eigenvalues of $A$-divergent form elliptic operators:
\begin{equation}\label{EllDivOper}
L_{A}=-\textrm{div} [A(w) \nabla g(w)], \quad w=(u,v)\in \Omega, \quad \left\langle A(w) \nabla g, n \right\rangle|_{\partial \Omega}=0,
\end{equation}
with $A \in M^{2 \times 2}(\Omega)$. We denote, by $M^{2 \times 2}(\Omega)$, the class of all $2 \times 2$ symmetric matrix functions $A(w)=\left\{a_{kl}(w)\right\}$, $\textrm{det} A=1$, with measurable entries satisfying the uniform ellipticity condition
\begin{equation}\label{UEC}
\frac{1}{K}|\xi|^2 \leq \left\langle A(w) \xi, \xi \right\rangle \leq K |\xi|^2 \,\,\, \text{a.e. in}\,\,\, \Omega,
\end{equation}
for every $\xi \in \mathbb C$, where $1\leq K< \infty$.
Such type elliptic operators in divergence form arise in various problems of mathematical physics (see, for example, \cite{AIM}).

The suggested approach is based on the quasiconformal change of variables that allows to reduce (\ref{EllDivOper}) to the weak weighted eigenvalue problem in $\mathbb D$:
\begin{equation}\label{WFWEP}
\iint\limits_\mathbb D \left\langle \nabla f(z), \nabla \overline{g(z)} \right\rangle dxdy
= \mu \iint\limits_\mathbb D h(z)f(z)\overline{g(z)}~dxdy, \,\,\, \forall g\in W^{1,2}(\mathbb D). 
\end{equation}

On this base we obtain the spectral stability estimates in a wide class of $A$-quasiconformal $\beta$-regular domains $\Omega\subset\mathbb C$, or in another terminology, domain satisfy quasihyperbolic boundary condition \cite{KOT}.

In \cite{GPU2019} (see also \cite{KOT,M}) it was proved that in $A$-quasiconformal regular domains $\Omega \subset \mathbb C$ the embedding operator 
\[
i_{\Omega}:W^{1,2}(\Omega,A) \hookrightarrow L^2(\Omega)
\]
is compact. 

Hence we conclude that in $A$-quasiconformal regular domains $\Omega \subset \mathbb C$ the Neumann spectrum of the divergence form elliptic operators $-\textrm{div} [A(w) \nabla g(w)]$ is discrete and can be written in the form of a non-decreasing sequence
\[
0=\mu_1[A,\Omega] < \mu_2[A,\Omega] \leq \ldots \leq \mu_n[A,\Omega] \leq \ldots,
\] 
where each eigenvalue is repeated as many time as its multiplicity (see, for example, \cite{M}).

Recall that a simply connected domain ${\Omega} \subset \mathbb C$ is called an $A$-quasiconformal $\beta$-regular domain \cite{GPU19} if
$$
\iint\limits_{{\Omega}} |J(w, \varphi)|^{1-\beta}~dudv < \infty, \,\,\,\beta>1,
$$
where $J(w, \varphi)$ is a Jacobian of an $A$-quasiconformal mapping $\varphi: {\Omega}\to\mathbb D$.

It is said that two $A$-quasiconformal $\beta$-regular domains $\Omega_1,\,\Omega_2$ represent an \textit{$A$-conformal $\beta$-regular pair} if there exists a conformal mappings $\psi:\Omega_1 \to \Omega_2$ such that
\[
\iint\limits_{\Omega_1}|J(z,\psi)|^{\beta}dxdy<\infty \,\, \& \,\,
\iint\limits_{\Omega_2}|J(w,\psi^{-1})|^{\beta}dudv<\infty.
\] 

Note, that two $A$-quasiconformal $\beta$-regular domains $\Omega_1=\varphi_1^{-1}(\mathbb D)$ and $\Omega_2=\varphi_2^{-1}(\mathbb D)$ represent an $A$-conformal $\beta$-regular pair if and only if 
\[
\Phi_{\beta}(\varphi_1,\varphi_2)=\left(\iint \limits_{\mathbb D}\max \left\{\frac{|J(z,\varphi_1^{-1}|^{\beta}}{|J(z,\varphi_2^{-1}|^{\beta -1}}, \frac{|J(z,\varphi_2^{-1}|^{\beta}}{|J(z,\varphi_1^{-1}|^{\beta -1}}\right\}
\right)^{\frac{1}{2\beta}}<\infty.
\] 
As an example, two Ahlfors type domains also represent an $A$-conformal regular pair.

The spectral stability estimates of the self-adjoint elliptic operators were intensively studied during the last decade. See, for example, \cite{BBL, BL1, BL2, BL3, BLLdeC, LP, LMS} where the history of the problem and main results in this area can be found.

In  the previous works \cite{BGU15, BGU16}, using an approach which is based on the conformal theory of composition operators on Sobolev spaces were established the conformal spectral stability estimates of Dirichlet eigenvalues and Neumann eigenvalues of Laplacian in non-Lipschitz domains and even in fractal type domains. In \cite{GPU2019a} were obtained the spectral stability estimates of Dirichlet eigenvalues of elliptic operators in divergence form.   

In the present paper we prove that if $\Omega_1=\varphi_1^{-1}(\mathbb D)$ and $\Omega_2=\varphi_2^{-1}(\mathbb D)$
represent an $A$-conformal $\beta$-regular pair, then for any $n\in \mathbb N$:
\[
|\mu_n[A, \Omega_1]-\mu_n[A, \Omega_2]| 
\leq c_n B^2_{\frac{4\beta}{\beta -1},2}(\mathbb D,h) \Phi_{\beta}(\varphi_1,\varphi_2) \cdot
\|J_{\varphi_1^{-1}}^{\frac{1}{2}}-J_{\varphi_2^{-1}}^{\frac{1}{2}}\,|\,L^{2}(\mathbb D)\|,
\]
where
$c_n=\max\left\{\mu_n^2[A, \Omega_1], \mu_n^2[A, \Omega_2]\right\}$,  $J_{\varphi_k^{-1}}$ are Jacobians of $A^{-1}$-quasiconformal mappings $\varphi_k^{-1}:\mathbb D\to\Omega_k$, $k=1,2$.
\begin{remark}
The constant $B_{\frac{4\beta}{\beta -1},2}(\mathbb D,h)$ is the best constant for corresponding weighted Poincar\'e-Sobolev inequalities in the unit disc $\mathbb D$.
\end{remark}

The method proposed to investigation the weak weight eigenvalue problem for the Neumann Laplacian is based on the Sobolev embedding theorems \cite{GG94,GU09} in connection with the composition operators on Sobolev spaces \cite{U93, VU02}.

\section{\textbf{SOBOLEV SPACES AND $A$-QUASICONFORMAL MAPPINGS}}

Let $\Omega \subset \mathbb C$ be a domain and $h:\Omega \to \mathbb R$ be a positive a.e. locally integrable function i.e. a weight. We consider the two-weighted Sobolev space $W^{1,p}(\Omega,h,1)$, $1\leq p< \infty$, defined
as the normed space of all locally integrable weakly differentiable functions
$f:\Omega\to\mathbb{R}$ endowed with the following norm:
\[
\|f\mid W^{1,p}(\Omega,h,1)\|=\|f\,|\,L^{p}(\Omega,h)\|+\|\nabla f\mid L^{p}(\Omega)\|.
\]

The weighted seminormed Sobolev space $L^{1,2}(\Omega, A)$ (associated with the matrix $A$), defined 
as the space of all locally integrable weakly differentiable functions $f:\Omega\to\mathbb{R}$
with the finite seminorm given by:  
\[
\|f\mid L^{1,2}(\Omega,A)\|=\left(\iint\limits_\Omega \left\langle A(z)\nabla f(z),\nabla f(z)\right\rangle\,dxdy \right)^{\frac{1}{2}}.
\]

The corresponding  normed Sobolev space $W^{1,2}(\Omega, A)$ is defined
as the normed space of all locally integrable weakly differentiable functions
$f:\Omega\to\mathbb{R}$ endowed with the following norm:
\[
\|f\mid W^{1,2}(\Omega, A)\|=\|f\,|\,L^{2}(\Omega)\|+\|f\mid L^{1,2}(\Omega,A)\|.
\]

Recall that a homeomorphism $\varphi: \Omega\to \Omega'$, $\Omega,\, \Omega'\subset\mathbb C$, is called a $K$-quasiconformal mapping if $\varphi\in W^{1,2}_{\loc}(\Omega)$ and there exists a constant $1\leq K<\infty$ such that
$$
|D\varphi(w)|^2\leq K |J(w,\varphi)|\,\,\text{for almost all}\,\,w\in\Omega.
$$

Now we give a construction of $A$-quasiconformal mappings connected with the $A$-divergent form elliptic operators.
Let $A \in M^{2 \times 2}(\Omega)$, consider the Beltrami equation:
\begin{equation}\label{BelEq}
\varphi_{\overline{w}}(w)=\mu(w) \varphi_{w}(w),\,\,\, \text{a.e. in}\,\,\, \Omega,
\end{equation}
with the complex dilatation $\mu(w)$ is given by
\begin{equation}\label{ComDil}
\mu(w)=\frac{a_{22}(w)-a_{11}(w)-2ia_{12}(w)}{\det(I+A(w))},\quad I= \begin{pmatrix} 1 & 0 \\ 0 & 1 \end{pmatrix}.
\end{equation}

Then the uniform ellipticity condition \eqref{UEC} can be written as
\begin{equation}\label{OVCE}
|\mu(w)|\leq \frac{K-1}{K+1},\,\,\, \text{a.e. in}\,\,\, \Omega.
\end{equation}

Conversely we can obtain from \eqref{ComDil} (see, for example, \cite{AIM}, p. 412) that :
\begin{equation}\label{Matrix-F}
A(w)= \begin{pmatrix} \frac{|1-\mu|^2}{1-|\mu|^2} & \frac{-2 \Imag \mu}{1-|\mu|^2} \\ \frac{-2 \Imag \mu}{1-|\mu|^2} &  \frac{|1+\mu|^2}{1-|\mu|^2} \end{pmatrix},\,\,\, \text{a.e. in}\,\,\, \Omega.
\end{equation}

So, given any $A \in M^{2 \times 2}(\Omega)$, one produced, by \eqref{OVCE}, the complex dilatation $\mu(w)$, for which, in turn, the Beltrami equation \eqref{BelEq} induces a quasiconformal homeomorphism $\varphi:\Omega \to \varphi(\Omega)$ as its solution, by the Riemann measurable mapping theorem (see, for example, \cite{Ahl66}). We will say that the matrix function $A$ induces the corresponding $A$-quasiconformal mappings $\varphi$ or that $A$ and $\varphi$ are agreed \cite{GNR18}. The $A$-quasiconformal mapping $\psi:\Omega\to\mathbb D$ of simply connected domain $\Omega\subset\mathbb C$ onto the unit disc $\mathbb D\subset\mathbb C$ can be obtained as a composition of $A$-quasiconformal homeomorphism $\varphi:\Omega\to\varphi(\Omega)$ and a conformal mapping $\omega:\varphi(\Omega)\to\mathbb D$.

So, by the given $A$-divergent form elliptic operator defined in a domain $\Omega\subset\mathbb C$ we construct so-called a $A$-quasiconformal mapping $\psi:\Omega\to\mathbb D$ with a quasiconformal coefficient
$$
K=\frac{1+\|\mu\mid L^{\infty}(\Omega)\|}{1-\|\mu\mid L^{\infty}(\Omega)\|},
$$
where $\mu$ is defined by (\ref{ComDil}).

Note that the inverse mapping to the $A$-quasiconformal mapping $\psi: \Omega \to \mathbb D$ is the $A^{-1}$-quasiconformal mapping \cite{GPU2019}.

In \cite{GPU2019} was given a connection between composition operators on Sobolev spaces and $A$-quasiconformal mappings. 

\begin{theorem}\label{L4.1}
Let $\Omega,\Omega'$ be domains in $\mathbb C$. Then a homeomorphism $\varphi :\Omega \to \Omega'$ is an $A$-quasiconformal mapping 
if and only if $\varphi$ induces, by the composition rule $\varphi^{*}(f)=f \circ \varphi$,
an isometry of Sobolev spaces $L^{1,2}(\Omega,A)$ and $L^{1,2}(\Omega')$:
\[
\|\varphi^{*}(f)\,|\,L^{1,2}(\Omega,A)\|=\|f\,|\,L^{1,2}(\Omega')\|
\]
for any $f \in L^{1,2}(\Omega')$.
\end{theorem}

This theorem generalizes the well known property of conformal mappings generate the isometry of uniform Sobolev spaces $L^1_2(\Omega)$ and $L^1_2(\Omega')$ (see, for example, \cite{C50}) and refines (in the case $n=2$) the functional characterization of quasiconformal mappings in the terms of isomorphisms of uniform Sobolev spaces \cite{VG75}.

\section{\textbf{THE WEIGHTED EIGENVALUE PROBLEM}}

In this section, using Theorem~\ref{L4.1} we prove the weighted Poincar\'e-Sobolev inequality in the unit disc $\mathbb D$ for so-called quasihyperbolic weights which are Jacobians of mappings inverse to $A$-quasiconformal homeomorphisms.

Denote by $h=h(z):=|J(z,\varphi^{-1})|$ the quasihyperbolic (quasiconformal) weight generated by the $A^{-1}$-quasiconformal mapping $\varphi:\mathbb D\to\Omega$ and 
\[
f_{\mathbb D,h}=\frac{1}{m_h(\mathbb D)} \iint\limits_{\mathbb D} f(z)h(z)dxdy=g_{\Omega}=\frac{1}{|\Omega|}
\iint\limits_\Omega g(w)dudv,
\]
\[
f(z)=g(\varphi(z)),\,\,\, w=\varphi(z).
\]
Here 
\[
m_h(\mathbb D)=\iint\limits_{\mathbb D} h(z)dxdy=\iint\limits_{\mathbb D} |J(z,\varphi^{-1})|dxdy=|\Omega|.
\]

Recall that for $A$-quasiconformal regular domains the Poincar\'e-Sobolev inequality
\begin{equation*}
\|g-g_{\Omega}\,|\,L^{2}(\Omega)\|\leq V^{*}\|g\,|\,L^{1,2}(\Omega, A)\|
\end{equation*}
holds for any function $g\in W^{1,2}(\Omega, A)$ with exact constant $V^{*}$ (see \cite{GPU2019}).

Given this inequality we prove: 
\begin{theorem}\label{WPIn}
Let $A$ be a matrix satisfies the uniform ellipticity condition~\eqref{UEC} and $\Omega \subset \mathbb C$ be an $A$-quasiconformal $\beta$-regular domain.
Then the weighted embedding operator
\begin{equation}\label{EmbedOper}
i_{\mathbb D}:W^{1,2}(\mathbb D) \hookrightarrow L^2{(\mathbb D,h)}
\end{equation}
is compact and for any function $f\in W^{1,2}(\mathbb D, h, 1)$ the inequality 
\[
\|f-f_{\mathbb D, h}\,|\,L^{2}(\mathbb D,h)\|\leq V^{*}\|f\,|\,L^{1,2}(\mathbb D)\|
\]
holds.
\end{theorem}

\begin{proof} Define the complex dilatation $\mu(z)$ agreed with the matrix $A$ by 
\begin{equation*}
\mu(z)=\frac{a_{22}(z)-a_{11}(z)-2ia_{12}(z)}{\det(I+A(z))}.
\end{equation*} 
Because the matrix $A$ satisfies the uniform ellipticity condition (\ref{UEC}) then 
\begin{equation*}
|\mu(w)|\leq \frac{K-1}{K+1}<1,\,\,\, \text{a.e. in}\,\,\, \Omega,
\end{equation*}
and by \cite{Ahl66} there exists a $\mu$-quasiconformal mapping $\varphi : \Omega \to \mathbb D$ agreed with the matrix $A$
i.e. an $A$-quasiconformal mapping. 
Hence by Theorem~\ref{L4.1} the composition operator
\[
\varphi^{*}:L^{1,2}(\mathbb D)\to L^{1,2}(\Omega, A),\,\,\, \varphi^{*}(f)= f\circ \varphi
\]
is an isometry. 

Let $f\in L^{1,2}(\mathbb D)$ be a smooth function.
Then the composition $g=f\circ \varphi$ belongs to $L^{1,2}(\Omega, A)$ and because the matrix $A$ satisfies the uniform ellipticity condition (\ref{UEC}) by the Sobolev embedding theorem we can conclude that $g=f\circ \varphi\in W^{1,2}(\Omega, A)$ \cite{M} and the Poincar\'e-Sobolev inequality
\begin{equation}\label{PSIn10}
\|g-g_{\Omega}\,|\,L^{2}(\Omega)\|\leq V^{*}\|g\,|\,L^{1,2}(\Omega, A)\|
\end{equation}
holds with the exact constant $V^{*}=\mu_1[A,\Omega]^{-\frac{1}{2}}$.

Now using the ``transfer" diagram \cite{GG94,GU09} and the change of variable formula for quasiconformal mappings \cite{VGR} we obtain
\begin{multline*}
\|f-f_{\mathbb D, h}\,|\,L^{2}(\mathbb D,h)\| \\
=\left(\iint \limits_{\mathbb D} |f(z)-f_{\mathbb D, h}|^2h(z)dxdy \right)^{\frac{1}{2}}
=\left(\iint \limits_{\mathbb D} |f(z)-f_{\mathbb D, h}|^2 |J(z,\varphi^{-1})|dxdy \right)^{\frac{1}{2}} \\
=\left(\iint \limits_{\Omega} |g(w)-g_{\Omega}|^2dudv \right)^{\frac{1}{2}}
\leq V^{*}\left(\iint \limits_{\Omega} \left\langle A(w) \nabla g(w), \nabla g(w) \right\rangle dudv \right)^{\frac{1}{2}} \\
=V^{*} \left(\iint \limits_{\mathbb D} |\nabla f(z)|^2dxdy \right)^{\frac{1}{2}}
=V^{*}\|f\,|\,L^{1,2}(\mathbb D)\|.
\end{multline*} 
Approximating an arbitrary function  $f \in W^{1,2}(\mathbb D,h,1)$ by smooth functions we have that the weighted Poincar\'e-Sobolev inequality
\[
\|f-f_{\mathbb D, h}\,|\,L^{2}(\mathbb D,h)\| \leq V^{*}\|f\,|\,L^{1,2}(\mathbb D)\|
\] 
holds for any function $f \in W^{1,2}(\mathbb D,h,1)$. 

Further we prove that the embedding operator 
\begin{equation}
i_{\mathbb D}:W^{1,2}(\mathbb D) \hookrightarrow L^2{(\mathbb D,h)}
\end{equation}
is compact.  By the same "transfer" diagram \cite{GG94,GU09} this operator can be represented as a composition of three operators: the composition operator $\varphi^{*}_w :W^{1,2}(\mathbb D)\to W^{1,2}(\Omega,A)$, the compact embedding operator 
\[i_{\Omega}:W^{1,2}(\Omega,A) \hookrightarrow L^2(\Omega)
\] 
and the composition operator for Lebesgue spaces $(\varphi^{-1})^{*}_l:L^2(\Omega) \to L^2(\mathbb D, h)$. 

Firstly we prove that the operator $(\varphi^{-1})^{*}_l$ is an isometry. By the change of variables formula we obtain:
\begin{multline*}
\|f\,|\,L^{2}(\mathbb D,h)\|
=\left(\iint \limits_{\mathbb D} |f(z)|^2h(z)~dxdy \right)^{\frac{1}{2}} \\
=\left(\iint \limits_{\mathbb D} |f(z)|^2 |J(z,\varphi^{-1})|~dxdy \right)^{\frac{1}{2}} 
=\left(\iint \limits_{\Omega} |f\circ \varphi(w)|^2~dudv \right)^{\frac{1}{2}}
=\|g\,|\,L^{2}(\Omega)\|.
\end{multline*}

Secondly we prove that the composition operator
\[
(\varphi^{-1})_w^{*}:W^{1,2}(\mathbb D)\to W^{1,2}(\Omega, A)
\]
is bounded.

By Theorem~\ref{L4.1} $A$-quasiconformal mappings $\varphi: \Omega \to \mathbb D$ generate a bounded composition operator on seminormed Sobolev spaces
\[
\varphi^{*}:L^{1,2}(\mathbb D) \to L^{1,2}(\Omega, A).
\]
Since the matrix $A$ satisfies to the uniform ellipticity condition~\eqref{UEC} then the norm of Sobolev space $W^{1,2}(\Omega, A)$ is equivalent to the norm of Sobolev space $W^{1,2}(\Omega)$ and by \cite{GPU19} we obtain that the composition operator on normed Sobolev spaces
\[
\varphi_w^{*}:W^{1,2}(\mathbb D)\to W^{1,2}(\Omega, A)
\]
is bounded.

Hence the embedding operator $i_{\mathbb D}$ is compact as a composition of the compact operator $i_{\Omega}$ and bounded operators $\varphi^{*}_w$ and $(\varphi^{-1})^{*}_l$.

\end{proof}

According to Theorem~\ref{WPIn} the weighted embedding operator is compact. By standard arguments we conclude that the spectrum of the weighted eigenvalue problem~\eqref{WFWEP} with quasihyperbolic (quasiconformal)
weights $h$ is discrete and can be written in the form of a non-decreasing sequence
\[
0=\mu_1[h,\mathbb D] < \mu_2[h,\mathbb D] \leq \ldots \leq \mu_n[h,\mathbb D] \leq \ldots,
\] 
where each eigenvalue is repeated as many time as its multiplicity (see, for example, \cite{B72, GP2009}).
The weighted eigenvalue problem in the unit disc $\mathbb D$ is equivalent to the eigenvalue problem in the domain $\Omega$ and 
\begin{equation}\label{EqEigenvalues}
\mu_n[h,\mathbb D]=\mu_n[A, \Omega],\,\,\, n\in \mathbb N.
\end{equation}
For weighted eigenvalues we have the following properties \cite{BGU16}:

($i$)  $\lim \limits_{n \to \infty} \mu_n[h,\mathbb D]=\infty$.

($ii$) for each $n\in \mathbb N$
\begin{multline}\label{MinMax}
\mu_n[A, \Omega]= \inf \limits_{\substack{L \subset W^{1,2}(\Omega,A)\\ \dim L=n}}
\sup \limits_{\substack{g\in L \\ g \neq 0}} 
\frac{\iint\limits_{\Omega} \left\langle A(w) \nabla g, \nabla g\right\rangle dudv}{\iint\limits_{\Omega} |g|^2dudv} \\
= \inf \limits_{\substack{L \subset W^{1,2}(\mathbb D,h,1)\\ \dim L=n}}
\sup \limits_{\substack{f\in L \\ f \neq 0}} 
\frac{\iint\limits_{\mathbb D} |\nabla f|^2dxdy}{\iint\limits_{\mathbb D} |f|^2h(z)dxdy}=\mu_n[h,\mathbb D]
\end{multline}
(Min-Max Principle), and
\begin{equation}\label{MaxPr}
\mu_n[h,\mathbb D]= \sup \limits_{\substack{f\in M_n \\ f \neq 0}} 
\frac{\iint\limits_{\mathbb D} |\nabla f|^2dxdy}{\iint\limits_{\mathbb D} |f|^2h(z)dxdy}
\end{equation} 
where
\[
M_n= \spn \{\varphi_1[h], \ldots , \varphi_n[h]\}
\]
and $\{\varphi_n[h]\}_{n=1}^{\infty}$ is an orthonormal (in the space $W^{1,2}(\mathbb D,h,1)$) set of eigenfunctions corresponding to the eigenvalues $\{\mu_n[h,\mathbb D]\}_{n=1}^{\infty}$.

($iii$) $\mu_1[h,\mathbb D]=0$ and $\varphi_1=\frac{1}{\sqrt{m_h(\mathbb D)}}$.
For $n \geq 2$ alongside with \eqref{MaxPr} we have
\begin{equation}\label{MaxPr-1}
\mu_n[h,\mathbb D]= \sup \limits_{\substack{f\in M_n \\ f \neq 0}} 
\frac{\iint\limits_{\mathbb D} |\nabla f|^2dxdy}{\iint\limits_{\mathbb D} |f-f_{\mathbb D,h}|^2h(z)dxdy}.
\end{equation} 
(It may happen that the above fraction takes the form $\frac{0}{0}$. In this case we assume that $\frac{0}{0}=0$.)

\section{\textbf{THE $L^{1,2}$-SEMINORM ESTIMATES}}

In this section we estimate variation of Neumann eigenvalues of two weighted eigenvalue problems in the unit disc $\mathbb D$:
\[
\iint\limits _{\mathbb D} \left\langle \nabla f(z), \nabla\overline{{g(z)}}\right\rangle~dxdy=\mu\iint\limits _{\mathbb D}h_{1}(z)f(z)\overline{g(z)}~dxdy\,,\,\,~~\forall g\in W^{1,2}(\mathbb D,h,1)
\]
 and
\[
\iint\limits _{\mathbb D} \left\langle \nabla f(z), \nabla\overline{{g(z)}}\right\rangle~dxdy=\mu\iint\limits _{\mathbb D}h_{2}(z)f(z)\overline{{g(z)}}~dxdy\,,\,\,~~\forall g\in W^{1,2}(\mathbb D,h,1).
\]

The following result in the case of hyperbolic (conformal) weights was proved in (\cite{BGU16}, Lemma 3.1). In the present paper we formulate this lemma in the case of quasihyperbolic (quasiconformal) weights.

\begin{lemma}
 \label{lem:TwoWeight} Let $\mathbb D \subset \mathbb{C}$ be the unit disc
and let $h_{1}$, $h_{2}$ be quasiconformal weights on $\mathbb D$.
Suppose that there exists a constant $B>0$ such that
\begin{multline}
\max \left\{\iint\limits _{\mathbb D}|h_{1}(z)-h_{2}(z)||f-f_{\mathbb D,h_1}|^{2}~dxdy,
\iint\limits _{\mathbb D}|h_{1}(z)-h_{2}(z)||f-f_{\mathbb D,h_2}|^{2}~dxdy\right\}\\
\leq B\iint\limits _{\mathbb D}|\nabla f|^{2}~dxdy,\,\,
\forall f\in L^{1,2}(\mathbb D).\label{EqvWW}
\end{multline}

Then for any $n\in\mathbb{N}$
\begin{equation}
|\mu_{n}[h_{1},\mathbb D]-\mu_{n}[h_{2},\mathbb D]|\leq \frac{B\tilde c_n}{1+B\sqrt{\tilde c_n}}< B\tilde c_n\,,\label{LemEq}
\end{equation}
where
\begin{equation}\label{tilde c_n}
\tilde c_n=\max\{\mu^2_{n}[h_{1},\mathbb D],\mu^2_{n}[h_{2},\mathbb D]\} \,.
\end{equation}
\end{lemma}

\vskip 0.2cm 

Further we estimate the constant $B$ in Lemma \ref{EqvWW} in terms of an $L^s$-distance between weights. Similarly to Theorem \ref{WPIn} we have
\begin{equation}\label{InPS}
\|f-f_{\mathbb D,h_k} \mid L^{r}(\mathbb D,h_k)\| \leq B_{r,2}(\mathbb D,h_k) \|\nabla f \mid L^{2}(\mathbb D)\|
\end{equation}
for $r \geq 1$ and any function $f \in W^{1,2}(\mathbb D,h_k,1)$, where $h_k=|J(z,\varphi_k^{-1})|$, $k=1,2$, are the quasiconformal weights defined by inverse mappings to $A$-quasiconformal homeomorphisms $\varphi_k:\Omega_k \to \mathbb D$. Here $B_{r,2}(\mathbb D,h_k)$, $k=1,2$, are best positive constants in these inequalities.

\begin{lemma}\label{lem:TwoWeiPol} 
Let   $h_{1}$, $h_{2}$ be quasiconformal weights on
$\mathbb{D}$ such that
\begin{equation}
d_{s}(h_{1},h_{2}):= \left\|(h_1-h_2) \left(\min\{h_1,h_2\}\right)^{\frac{1-s}{s}}\mid L^{s}(\mathbb{D})\right\|<\infty  \label{EqvWWpol}
\end{equation}
 for some $1<s\le\infty$.

Then  inequality $(\ref{EqvWW})$ holds with the constant
\begin{equation}
B=B_{\frac{2s}{s-1},2}^2(\mathbb D,h)\,d_{s}(h_{1},h_{2}), \label{Lem2Es}
\end{equation}
where
\[
B_{\frac{2s}{s-1},2}(\mathbb D,h)=\max\left\{B_{\frac{2s}{s-1},2}(\mathbb D,h_1),B_{\frac{2s}{s-1},2}(\mathbb D,h_2)\right\}.
\]
\end{lemma}

\begin{proof}
By the H\"older inequality and Poincar\'e-Sobolev inequality \eqref{InPS} we get for $k=1,2$
\begin{multline*}
\iint\limits _{\mathbb{D}}|h_{1}(z)-h_{2}(z)||f(z)-f_{\mathbb D,h_k}|^{2}~dxdy
\\
=\iint\limits _{\mathbb{D}}|h_{1}(z)-h_{2}(z)|h_k^{\frac{1-s}{s}} h_k^{\frac{s-1}{s}}|f(z)-f_{\mathbb D,h_k}|^{2}~dxdy  \\
\leq \left\|(h_1-h_2) h_k^{\frac{1-s}{s}}\mid L^{s}(\mathbb{D})\right\|
\left(\iint\limits _{\mathbb{D}}h_{k}(z)|f(z)-f_{\mathbb D,h_k}|^{\frac{2s}{s-1}}~dxdy\right)^{\frac{s-1}{s}} \\
\leq B_{\frac{2s}{s-1},2}^2(\mathbb D,h)\,d_{s}(h_{1},h_{2})
\iint\limits _{\mathbb{D}}|\nabla f(z)|^2~dxdy.
\end{multline*}
\end{proof}

By the two previous lemmas we have the following result for variations of the weighted eigenvalues:
\begin{theorem}\label{thm:TwoWW} 
Let  $h_{1}$, $h_{2}$ be quasiconformal weights on
$\mathbb{D}$. Assume that $d_{s}(h_{1},h_{2})<\infty$ for some $s>1$.

Then for any $n\in\mathbb{N}$
\[
|\mu_{n}[h_{1},\mathbb D]-\mu_{n}[h_{2}, \mathbb D]|\leq \tilde c_{n}B_{\frac{2s}{s-1},2}^2(\mathbb D,h) d_{s}(h_{1},h_{2})\,.
\]
\end{theorem}

\section{\textbf{ON ``\,DISTANCE\," $d_s(h_1, h_2)$ FOR QUASICONFORMAL WEIGHTS $h_1(z)$, $h_2(z)$}}

Let $\Omega_{k}$, $k=1,2$, be bounded simply connected domains in $\mathbb C$ and $A\in M^{2 \times 2}(\Omega_{k})$. Assume that there exist $A$-quasiconformal mappings $\varphi_1:\Omega_1 \to \mathbb D$ and $\varphi_2:\Omega_2 \to \mathbb D$.
Recall that for quasiconformal mappings $\varphi^{-1}:\mathbb D\to\Omega$
$$
J_{\varphi^{-1}}(z):=\lim\limits_{r\to 0}\frac{|\varphi^{-1}(B(z,r))|}{|B(z,r)|}=|J(z, \varphi^{-1})|
$$
for almost all $z\in\mathbb D$. 

Now we estimate the quantity $d_{s}(h_{1},h_{2})$ using the $L^2$-norms of weights.
\begin{lemma}\label{Cor-5.1}
Let $\varphi_1:\Omega_1 \to \mathbb D$ and $\varphi_2:\Omega_2 \to \mathbb D$
be $A$-quasiconformal homeomorphisms and $h_1$, $h_2$ be the corresponding
quasihyperbolic weights. Assume that for some $\beta >1$
\[
\Phi_{\beta}(\varphi_1,\varphi_2)=\left(\iint \limits_{\mathbb D}\max \left\{\frac{|J(z,\varphi_1^{-1})|^{\beta}}{|J(z,\varphi_2^{-1}|^{\beta -1})}, 
\frac{|J(z,\varphi_2^{-1})|^{\beta}}{|J(z,\varphi_1^{-1})|^{\beta -1}}\right\}
\right)^{\frac{1}{2\beta}}<\infty.
\]

Then for $s=\frac{2\beta}{\beta +1}$
\[
d_{s}(h_{1},h_{2}) \leq 2 \Phi_{\beta}(\varphi_1,\varphi_2) \cdot
\left(\iint\limits_{\mathbb{D}}\left(|J(z,\varphi_1^{-1})|^{\frac{1}{2}}-|J(z,\varphi_2^{-1})|^{\frac{1}{2}}\right)^2dxdy\right)^{\frac{1}{2}},
\]
where $J(z,\varphi_1^{-1}),\, J(z,\varphi_2^{-1})$ are Jacobians inverse mappings to $A$-quasiconformal homeomorphisms $\varphi_1:\Omega_1 \to \mathbb D$ and $\varphi_2:\Omega_2 \to \mathbb D$ respectively. 
\end{lemma}

\begin{proof}
By the definitions of $h_1$, $h_2$ and $d_s(h_1,h_2)$
\begin{multline*}
[d_s(h_1,h_2)]^s=\iint\limits_{\mathbb{D}}|h_{1}(z)-h_{2}(z)|^s (\min\{h_1(z),h_2(z)\})^{1-s}~dxdy \\
= \iint\limits_{\mathbb{D}}||J(z,\varphi_1^{-1})|-|J(z,\varphi_2^{-1})||^s (\min\{|J(z,\varphi_1^{-1})|,|J(z,\varphi_2^{-1})|\})^{1-s}~dxdy \\
= \iint\limits_{\mathbb{D}}\left||J(z,\varphi_1^{-1})|^{\frac{1}{2}}+|J(z,\varphi_2^{-1})|^{\frac{1}{2}}\right|^s (\min\{|J(z,\varphi_1^{-1})|,|J(z,\varphi_2^{-1})|\})^{1-s} \\
\times \left||J(z,\varphi_1^{-1})|^{\frac{1}{2}}-|J(z,\varphi_2^{-1})|^{\frac{1}{2}}\right|^s~dxdy.
\end{multline*}
Applying to the last integral the H\"older inequality with $q=\frac{2}{s}$ ($1\leq q<2$ because $1<s\leq 2$) and $q'=\frac{2}{2-s}$ we obtain
\begin{align*}
[d_s(h_1,h_2)]^s
&=\Bigg(\iint\limits_{\mathbb{D}}\left||J(z,\varphi_1^{-1})|^{\frac{1}{2}}+|J(z,\varphi_2^{-1})|^{\frac{1}{2}}\right|^{\frac{2s}{2-s}} \\
&\times (\min\{|J(z,\varphi_1^{-1})|,|J(z,\varphi_2^{-1})|\})^{\frac{2(1-s)}{2-s}}dxdy\Bigg)^{\frac{2-s}{2}} \\
&\times 
\left(\iint\limits_{\mathbb{D}}\left(|J(z,\varphi_1^{-1})|^{\frac{1}{2}}-|J(z,\varphi_2^{-1})|^{\frac{1}{2}}\right)^2dxdy\right)^{\frac{1}{2}} \\
&\leq 2^s \left(
\iint\limits_{\mathbb{D}} \frac{\max \left\{|J(z,\varphi_1^{-1})|^{\frac{1}{2}},|J(z,\varphi_2^{-1})|^{\frac{1}{2}}\right\}^{\frac{2s}{2-s}}}{\min \left\{|J(z,\varphi_1^{-1})|^{\frac{1}{2}},|J(z,\varphi_2^{-1})|^{\frac{1}{2}}\right\}^{\frac{4(s-1)}{2-s}}} dxdy \right)^{\frac{2-s}{2}} \\
&\times 
\left(\iint\limits_{\mathbb{D}}\left(|J(z,\varphi_1^{-1})|^{\frac{1}{2}}-|J(z,\varphi_2^{-1})|^{\frac{1}{2}}\right)^2dxdy\right)^{\frac{1}{2}}.
\end{align*}
Since $s=\frac{2\beta}{\beta +1}$ we have
\begin{align*}
d_s(h_1,h_2)  
&\leq 2 \left(
\iint\limits_{\mathbb{D}} \frac{\max \left\{|J(z,\varphi_1^{-1})|^{\frac{1}{2}},|J(z,\varphi_2^{-1})|^{\frac{1}{2}}\right\}^{2 \beta}}{\min \left\{|J(z,\varphi_1^{-1})|^{\frac{1}{2}},|J(z,\varphi_2^{-1})|^{\frac{1}{2}}\right\}^{2(\beta -1)}} dxdy \right)^{\frac{1}{2 \beta}} \\
&\times 
\left(\iint\limits_{\mathbb{D}}\left(|J(z,\varphi_1^{-1})|^{\frac{1}{2}}-|J(z,\varphi_2^{-1})|^{\frac{1}{2}}\right)^2dxdy\right)^{\frac{1}{2}} \\
&= 2 \left(\iint \limits_{\mathbb D}\max \left\{\frac{|J(z,\varphi_1^{-1}|^{\beta}}{|J(z,\varphi_2^{-1}|^{\beta -1}}, \frac{|J(z,\varphi_2^{-1}|^{\beta}}{|J(z,\varphi_1^{-1}|^{\beta -1}}\right\}
\right)^{\frac{1}{2\beta}} \\
&\times 
\left(\iint\limits_{\mathbb{D}}\left(|J(z,\varphi_1^{-1})|^{\frac{1}{2}}-|J(z,\varphi_2^{-1})|^{\frac{1}{2}}\right)^2dxdy\right)^{\frac{1}{2}}.
\end{align*}
\end{proof}

Let us mention some results for quasiconformal mappings (see, for example, \cite{AIM}). Let $\Omega_1,\,\Omega_2$ and $\Omega_3$ be bounded simply connected domains on the complex plane $\mathbb C$.

If mapping $\psi:\Omega_1 \to \Omega_2$ is quasiconformal, then
$|J(z, \psi)|=|J(w, \psi^{-1})|^{-1}$ for almost all $z\in \Omega_1$ and for almost all $w=\psi(z)\in \Omega_2$.

If mappings $\psi_1:\Omega_1 \to \Omega_2$ and $\psi_2:\Omega_2 \to \Omega_3$ are quasiconformal, then
\[
J(z,\psi)=J(z,\psi_1) \cdot J(\psi_1(z),\psi_2),\,\, \text{a.e. in}\,\, \Omega_1\,\, (\psi=\psi_2 \circ \psi_1(z)).
\]

If two quasiconformal mappings $\varphi:\Omega_1 \to \Omega_2$ and $\psi:\Omega_1 \to \Omega_3$, defined on $\Omega_1$, have the same Beltrami coefficient, then the mapping $\psi \circ \varphi^{-1}:\Omega_1 \to \Omega_2$ is conformal.

Given these results we prove the following assertion.

\begin{proposition}
Two $A$-quasiconformal $\beta$-regular domains $\Omega_1=\varphi_1^{-1}(\mathbb D)$ and $\Omega_2=\varphi_2^{-1}(\mathbb D)$ represent an $A$-conformal $\beta$-regular pair if and only if 
\[
\Phi_{\beta}(\varphi_1,\varphi_2)=\left(\iint \limits_{\mathbb D}\max \left\{\frac{|J(z,\varphi_1^{-1}|^{\beta}}{|J(z,\varphi_2^{-1}|^{\beta -1}}, \frac{|J(z,\varphi_2^{-1}|^{\beta}}{|J(z,\varphi_1^{-1}|^{\beta -1}}\right\}
\right)^{\frac{1}{2\beta}}<\infty.
\]
\end{proposition}

\begin{proof}
The conditions $\Phi_{\beta}(\varphi_1,\varphi_2)<\infty$ means that domains $\Omega_1=\varphi_1^{-1}(\mathbb D)$, $\Omega_2=\varphi_2^{-1}(\mathbb D)$ represent a $A$-conformal $\beta$-regular pair. Indeed, conformal mapping $\psi:\Omega_1 \to \Omega_2$ can be written as a composition
\[
\psi(w)=\varphi_2^{-1}\left(\varphi_1(w)\right).
\]
Then
\[
J(w,\psi)=J(w,\varphi_1(w))\cdot J(\varphi_1(w), \varphi_2^{-1}) ,\,\, \text{a.e. in}\,\, \Omega_1.
\]
Using the change of variable formula we obtain
\begin{align*}
\iint\limits_{\Omega_1} |J(w,\psi)|^{\beta}~dudv &=\iint\limits_{\mathbb D} |J(\varphi_1^{-1}(z),\psi)|^{\beta} |J(z,\varphi_1^{-1})|~dxdy \\
&=\iint\limits_{\mathbb{D}} |J(\varphi_1^{-1}(z),\varphi_1)|^{\beta}|J(z,\varphi_2^{-1}|^{\beta}|J(z,\varphi_1^{-1})|~dxdy \\
&=\iint\limits_{\mathbb{D}} |J(z,\varphi_1^{-1})|^{-\beta}|J(z,\varphi_2^{-1}|^{\beta}|J(z,\varphi_1^{-1})|~dxdy \\
&= \iint\limits_{\mathbb{D}} |J(z,\varphi_2^{-1})|^{\beta} |J(z,\varphi_1^{-1})|^{1-\beta}~dxdy.
\end{align*} 
The similar calculation is correct for the inverse mapping $\psi^{-1}:\Omega_2 \to \Omega_1$.
\end{proof}

Now we are ready to prove the main result of this work.
\begin{theorem}\label{Main}
Let $A$ be a matrix satisfies the uniform ellipticity condition \eqref{UEC} and $\Omega_1=\varphi_1^{-1}(\mathbb D)$, $\Omega_2=\varphi_2^{-1}(\mathbb D)$ be an $A$-conformal $\beta$-regular pair.
 
Then for any $n\in \mathbb N$
\[
|\mu_n[A, \Omega_1]-\mu_n[A, \Omega_2]| 
\leq 2c_n B^2_{\frac{4\beta}{\beta -1},2}(\mathbb D,h) \Phi_{\beta}(\varphi_1,\varphi_2) \cdot
\|J_{\varphi_1^{-1}}^{\frac{1}{2}}-J_{\varphi_2^{-1}}^{\frac{1}{2}}\,|\,L^{2}(\mathbb D)\|,
\]
where
$c_n=\max\left\{\mu_n^2[A, \Omega_1], \mu_n^2[A, \Omega_2]\right\}$ and
$J_{\varphi_1^{-1}}$, $J_{\varphi_2^{-1}}$ are Jacobians inverse mappings to $A$-quasiconformal homeomorphisms $\varphi_1:\Omega_1 \to \mathbb D$ and $\varphi_2:\Omega_2 \to \mathbb D$ respectively. 
\end{theorem} 

\begin{proof}
According to Theorem \ref{thm:TwoWW} we have
\[
|\mu_{n}[h_{1},\mathbb D]-\mu_{n}[h_{2}, \mathbb D]|\leq \tilde c_{n}B_{\frac{2s}{s-1},2}^2(\mathbb D,h) d_{s}(h_{1},h_{2})\,.
\]
By Lemma \ref{Cor-5.1} we obtain the estimate of quality $d_{s}(h_{1},h_{2})$. For $s=\frac{2 \beta}{\beta +1}$ we have
\[
d_{s}(h_{1},h_{2}) \leq 2 \Phi_{\beta}(\varphi_1,\varphi_2) \cdot
\|J_{\varphi_1^{-1}}^{\frac{1}{2}}-J_{\varphi_2^{-1}}^{\frac{1}{2}}\,|\,L^{2}(\mathbb D)\|.
\]
Finally, using equality \eqref{EqEigenvalues} and given that $\frac{2s}{s-1}=\frac{4\beta}{\beta -1}$ for $s=\frac{2 \beta}{\beta +1}$ we get the required inequality 
\[
|\mu_n[A, \Omega_1]-\mu_n[A, \Omega_2]| 
\leq 2c_n B^2_{\frac{4\beta}{\beta -1},2}(\mathbb D,h) \Phi_{\beta}(\varphi_1,\varphi_2) \cdot
\|J_{\varphi_1^{-1}}^{\frac{1}{2}}-J_{\varphi_2^{-1}}^{\frac{1}{2}}\,|\,L^{2}(\mathbb D)\|,
\]
\end{proof}

\begin{remark}
According to the work \cite{GPU19} the constant $B_{\frac{4\beta}{\beta -1},2}(\mathbb D,h)$ can be estimated as
\[
B_{\frac{4\beta}{\beta -1},2}(\mathbb D,h) \leq K^{\frac{1}{2}} B_{q,2}(\mathbb D) \max \left\{\|J_{\varphi_1^{-1}}\,|\,L^{2}(\mathbb D)\|^{\frac{\beta -1}{4 \beta}},
\|J_{\varphi_2^{-1}}\,|\,L^{2}(\mathbb D)\|^{\frac{\beta -1}{4 \beta}}\right\}
\] 
where $K$ is a quasiconformal coefficient of $\varphi_1,\, \varphi_2$ and 
\[
B_{q,2}(\mathbb D) \leq \left(2^{-1} \pi\right)^{\frac{2-q}{2q}}\left(q+2\right)^{\frac{q+2}{2q}},\,\, q=\left(\frac{2 \beta}{\beta -1}\right)^2.
\] 
\end{remark}

\section{\textbf{ON ISOSPECTRAL OPERATORS}}

Let us remind that two linear operators are called isospectral if they have the same spectrum. It is known that there are distinct domains such that all the eigenvalues of the linear operator (in the case of the Laplace operator, see, for instance, \cite{BCDS}) coincide. For this reason, these are called isospectral domains. 

Let $\varphi:\Omega \to \mathbb D$ be $A$-quasiconformal mappings for which $|J(w,\varphi)|=1$ for almost all $w \in \Omega$. In this case quasihyperbolic weights $h(z)=|J(z, \varphi^{-1})|=1$ for almost all $z \in \mathbb D$. Hence, by formula~\eqref{MinMax} we have  
$\mu_n[A, \Omega]= \mu_n[\mathbb D]$
So, we conclude that if $|J(w,\varphi)|=1$ for almost all $w \in \Omega$, then the Neumann eigenvalues for the elliptic operator in divergence form $-\textrm{div} [A(w) \nabla g(w)]$ in a domain $\Omega$ are equal to the Neumann eigenvalues for the Laplace operator in the unit disc $\mathbb D$.

\vskip 0.1cm

\textbf{Acknowledgements.} The first author was supported by the United States-Israel Binational Science Foundation (BSF Grant No. 2014055). The second author was supported by 
RFBR Grant No. 18-31-00011.

\begin{flushright}
\textit{Ben-Gurion University of the Negev,\newline
Department of Mathematics,\newline
P.O. Box 653, 8410501 Beer Sheva,\newline
Israel,\newline
vladimir@math.bgu.ac.il  \newline
}
\end{flushright}

\begin{flushright}
\textit{Tomsk Polytechnic University,\newline
Division for Mathematics and Computer Sciences,\newline
Lenin Ave. 30, 634050 Tomsk,\newline
Russia;\newline
Tomsk State University,\newline
International Laboratory SSP \& QF,\newline
Lenin Ave. 36, 634050 Tomsk,\newline
Russia;\newline
vpchelintsev@vtomske.ru \newline 
}
\end{flushright}

\begin{flushright}
\textit{Ben-Gurion University of the Negev,\newline
Department of Mathematics,\newline
P.O. Box 653, 8410501 Beer Sheva,\newline
Israel,\newline
ukhlov@math.bgu.ac.il  \newline
}
\end{flushright}

\end{document}